\documentclass{amsproc}
\usepackage{multirow}
\usepackage{array}

\usepackage{caption}
\usepackage{float}

\usepackage{xargs}
\usepackage{mathtools, todonotes}
\usepackage{amsmath,amssymb,amsfonts,amsthm,mathdots}

\usepackage[
  a4paper,
  left=27mm,
  right=27mm,
  top=27mm,
  bottom=30mm
]{geometry}

\usepackage[colorlinks,backref]{hyperref}
\usepackage{amsrefs}
\usepackage{enumerate}

\usepackage[onehalfspacing]{setspace}

\newcommandx{\set}[2][2=\empty]{\left\{#1\ifx#2\empty\else\,\middle\vert\,#2\fi\right\}}
\newcommandx{\tuple}[2][2=\empty]{(#1\ifx#2\empty\else\,|\,#2\fi)}
\newcommandx{\gensubgrp}[2][2=\empty]{\langle#1\ifx#2\empty\else\,|\,#2\fi\rangle}
\newcommandx{\gennorsubgrp}[2][2=\empty]{\langle\!\langle#1\ifx#2\empty\else\,|\,#2\fi\rangle\!\rangle}
\newcommandx{\gensubsp}[2][2=\empty]{\langle#1\ifx#2\empty\else\,|\,#2\fi\rangle}
\newcommand{\norm}[1]{\left\lvert#1\right\rvert}
\newcommand{\card}[1]{\left\lvert#1\right\rvert}
\DeclareMathOperator{\GL}{GL}
\DeclareMathOperator{\SL}{SL}
\DeclareMathOperator{\PSL}{PSL}

\DeclareMathOperator{\diag}{diag}

\DeclareMathOperator{\Sp}{Sp}
\DeclareMathOperator{\PSp}{PSp}
\DeclareMathOperator{\GO}{GO}

\DeclareMathOperator{\SO}{SO}

\DeclareMathOperator{\im}{im}

\title{Images of word maps with constants on algebraic groups}
\author{Vadim Alekseev}
\address{Vadim Alekseev, TU Dresden, 01062 Dresden, Germany}
\email{vadim.alekseev@tu-dresden.de}
\author{Jakob Schneider}
\address{Jakob Schneider, TU Dresden, 01062 Dresden, Germany}
\email{jakob.schneider@tu-dresden.de}


\theoremstyle{plain}
\newtheorem{theorem}{Theorem}
\newtheorem{lemma}[theorem]{Lemma}

\newtheorem{corollary}[theorem]{Corollary}

\newtheorem{fact}[theorem]{Fact}

\theoremstyle{definition}
\newtheorem{remark}[theorem]{Remark}
\newtheorem{definition}[theorem]{Definition}
\newtheorem{example}[theorem]{Example}

\begin{document}
\begin{abstract}
We study word maps with constants on quasisimple algebraic groups over a local field $L$. We prove that, for such a group $G=\mathbf{G}(L)$ and a word $w\in (G\ast\mathbf{F}_r)\setminus G$, either $w$ has a so-called Tomanov-small critical constant or the minimal dimension of the word image $w(G^r)\subseteq G$ of such a word is bounded from below by a function $c(G)>1$. We compute the optimal value of $c(G)$ for most quasisimple linear algebraic groups.
\end{abstract}
\maketitle

\vspace{6mm}

\section{Introduction and preliminaries}

The study of \emph{group laws} and \emph{word maps} has a long history; see, for example, Neumann's monograph~\cite{neumann1967varieties} and Borel's classical dominance theorem~\cite{borel1983free}, as well as the references in~\cite{gordeevkunyavskiiplotkin2018word}. Quantitative largeness of the image of a non-trivial word map was subsequently studied by Larsen~\cite{larsen2004word}. Related quantitative results on laws and ordinary word maps in finite groups were obtained by Bradford--Thom~\cites{bradfordthom2019short,bradfordthom2026short} and Schneider--Thom~\cite{schneiderthom2021word}. In this article we study similar questions for words with constants: Let $G\leq C$ be algebraic groups. Fix a word 
$$
w=w(x_1,\ldots,x_r)=c_0 x_{i(1)}^{\varepsilon(1)}c_1\cdots c_{l-1} x_{i(l)}^{\varepsilon(l)}c_l\in C\ast\mathbf{F}_r=C\ast\gensubgrp{x_1,\ldots,x_r}
$$ 
with constants from $C$. Then we obtain the \emph{word map} $w\colon G^r\to C$ by substitution $(g_1,\ldots,g_r)\mapsto w(g_1,\ldots,g_r)$. We call $w$ a \emph{mixed identity} for $G$ with constants in $C$ if $w(G^r)=\mathbf{1}$, i.e.\ if the word map is trivial. Mixed identities in classical groups and quasisimple linear algebraic groups were studied by Golubchik--Mikhalev~\cite{golubchikmikhalev1982generalized}, Tomanov~\cite{tomanov1985generalized}, and Gordeev~\cite{gordeev1997freedom}. More recent quantitative work on word maps with constants and mixed identities includes Schneider--Thom~\cite{schneiderthom2024constants} and Bradford--Schneider--Thom~\cites{bradfordschneiderthom2024nonsolutions,bradfordschneiderthom2025mixed,bradfordschneiderthom2023nonsingular}. In this article we make the results about absence of mixed identities quantitative by providing lower bounds on the \emph{dimension} (as a \emph{constructible set}, see Definition~\ref{def:dim_img}) of the word image $w(G^r)\subseteq C$ for words $w\in(C\ast\mathbf{F}_r)\setminus C$ which are not mixed identities and admit no so-called \emph{Tomanov-small} constants (see Definition~\ref{def:sml_consts} below).

\medskip

Let $w$ be \emph{reduced} as above, i.e. $i(j)=i(j+1)$ and $\varepsilon(j)=-\varepsilon(j+1)$ implies that $c_j\neq 1_C$. Call $c_1,\ldots,c_{l-1}$ the \emph{intermediate constants}. Define the following sets of indices:
\begin{align*}
J_0(w)&\coloneqq\set{j\in\set{1,\ldots,l-1}}[ i(j)\neq i(j+1)];\\ 
J_+(w)&\coloneqq\set{j\in\set{1,\ldots,l-1}}[ i(j)= i(j+1)\text{ and }\varepsilon(j)=\varepsilon(j+1)];\\
J_-(w)&\coloneqq\set{j\in\set{1,\ldots,l-1}}[ i(j)= i(j+1)\text{ and }\varepsilon(j)=-\varepsilon(j+1)].
\end{align*}
Call the elements of $J_-(w)$ \emph{critical indices} and the constants $c_j$ for $j\in J_-(w)$ \emph{critical constants}. The latter are collected in the set $I(w)\coloneqq\set{c_j}[j\in J_-(w)]$. Let the `small' subset $\Lambda\subseteq C$ given. Following \cite{golubchikmikhalev1982generalized}, we call a word $w$ \emph{$\Lambda$-non-central} if no critical constant of $w$ lies in $\Lambda$, i.e.\ $\Lambda\cap I(w)=\varnothing$. E.g.\ $w$ is $\mathbf{1}$-non-central if and only if it is reduced, which we assume in the following. Moreover, $w$ is $C$-non-central if and only if it has no critical constants. In this latter case, we call $w$ \emph{strict}. In the course of this article, we will need the following lemma, which is just a version of \cite{bradfordschneiderthom2025length}*{Lemma~2.2} but formulated for an infinite group. We leave its proof to the reader. This justifies that we only need to consider the case $r=1$, i.e.\ $w\in C\ast\gensubgrp{x}=C\ast\mathbb{Z}$.
	
\begin{lemma}\label{lem:red_one_var_cs}
Assume that $G$ is infinite and let $w\in C\ast\mathbf{F}_r=C\ast\gensubgrp{x_1,\ldots,x_r}$ be as above. Then there is a substitution $s\colon x_i\mapsto g_{-i} x g_i$ for $g_{\pm i}\in G$ ($i\in\set{1,\ldots,r}$) such that in 
$$
w'(x)\coloneqq w(s(x_1),\ldots,s(x_r))=c_0' x^{\varepsilon(1)} c_1' \cdots c_{l-1}' x^{\varepsilon(l)} c_l'\in C\ast\gensubgrp{x}
$$ 
we have $c_j'\neq 1_C$ for $j\in\set{1,\ldots,l-1}$. In this case, clearly $\im(w')\subseteq\im(w)$, so if $w$ is a mixed identity, we may assume that it has only on variable.
\end{lemma}
	
	
Fix an infinite field $K$ and a perfect linear algebraic $K$-group $\mathbf{G}\leq\GL_n$ acting absolutely irreducible on the vector space $\mathbf{V}\coloneqq\mathbf{A}^n$. Denote by $\mathbf{C}\leq\GL_n$ a reductive algebraic \emph{group of constants} containing $\mathbf{G}$. 
 Fix an algebraically closed overfield $F\geq K$. Let $C\coloneqq\mathbf{C}(F)$.
	
\begin{lemma}\label{lem:im_wrd_mp_constbl}
The word image $w(\mathbf{G}(F)^r)\subseteq\GL_n(F)$ is \emph{constructible} (i.e.\ a boolean combination of Zariski closed sets) and irreducible (i.e.\ it cannot be written as a proper union of two Zariski closed subsets of it).
\end{lemma}
	
\begin{proof}
Indeed, the word map
$$
\GL_n(F)^r\supseteq\mathbf{G}(F)^r\overset{w}{\twoheadrightarrow}w(\mathbf{G}(F)^r)\subseteq\mathbf{M}_n(F),
$$
is regular and Chevalley's theorem applies. Furthermore, images of irreducible sets under a continuous map are again irreducible.
\end{proof}	
	
Then the $K$-points of $\mathbf{G}$ are Zariski dense in $\mathbf{G}(F)$ (which holds by \cite{borel2012linear}*{Corollary~18.3}, since $\mathbf{G}$ is perfect). 
	
\begin{definition}\label{def:dim_img}
For a given subset $T\subseteq\mathbf{M}_n(F)$, we define the dimension $\dim(T)\in\mathbb{N}$ as the \emph{Krull dimension} of the Zariski closure $\overline{T}^F\subseteq\mathbf{M}_n(F)$.
\end{definition}

	
\begin{remark}\label{rmk:incr_fld}
Note here that
\begin{equation}\label{eq:wrd_im_dim}
\overline{w(\mathbf{G}(F)^r)}^F=\overline{w(\mathbf{G}(K)^r)}^F\subseteq\mathbf{M}_n(F),
\end{equation}
since $\mathbf{G}(K)^r$ is Zariski dense in $\mathbf{G}(F)^r$. Hence, replacing $K$ by an overfield, which is contained in $F$, does not change the dimension of $w(\mathbf{G}(K)^r)\subseteq\mathbf{M}_n(F)$. This will be exploited later in the text.
\end{remark}
	
Given a Zariski closed subset $\mathbf \Lambda$ of ``exceptional'' constants (in our case we'll set it to be the so-called \emph{Tomanov-small constants} \cite{tomanov1985generalized}, see Definition~\ref{def:sml_consts} below), we can define the following quantities:
$$
d(\mathbf{G},\mathbf{C},\mathbf{\Lambda})\coloneqq\min\set{\dim(w(\mathbf{G}(K)^r))}[w\in \mathbf{C}(F)\ast\mathbf{F}_r\setminus\mathbf{C}(F),\ \mathbf{\Lambda}(F)\cap I(w)=\varnothing];
$$
$$
d(\mathbf{G})\coloneqq\min\set{\dim(w(\mathbf{G}(K)^r))}[w\in \mathbf{G}(F)\ast\mathbf{F}_r\setminus\mathbf{G}(F)].
$$
Note that $d(\mathbf{G})=d(\mathbf{G},\mathbf{G},\mathbf{1})$. We have that $d(\mathbf{G})=0$ if and only if $\mathbf{G}(K)$ has a mixed identity $w\in\mathbf{G}(F)\ast\mathbf{F}_r$, since $\dim(w(\mathbf{G}(K)^r))=0$ implies (by Equation~\eqref{eq:wrd_im_dim}) that $w(\mathbf{G}(F)^r)$ is finite and connected and hence a singleton $\set{w(1_G,\ldots,1_G)}$, so that $w'\coloneqq ww(1_G,\ldots,1_G)^{-1}$ is a mixed identity for $\mathbf{G}(K)$. If $\mathbf{G}$ is not connected, then $\mathbf{G}(F)/\mathbf{G}^\circ(F)$ is a finite group of order say $e$ and the word map $w''(x)=x^e$ maps $\mathbf{G}(F)$ into $\mathbf{G}^\circ(F)$, so that $w(x_1^e,\ldots,x_r^e)$ maps $\mathbf{G}(F)\supseteq\mathbf{G}(K)$ to a singleton, and we proceed as before. However, in the following we shall assume that $\mathbf{G}$ is connected. In this article, we will compute the exact value of the quantity $d(\mathbf{G})$ for many quasisimple algebraic groups.
	
\section{Mixed identities for algebraic groups}
	
\subsection{Groups with no mixed identity are simple} 
	
At first we have the following structural result. Before we state it, we need a definition.
	
\begin{definition}\label{def:qsi_smpl_alg_grp}
The $K$-group $\mathbf{G}$, as introduced above, is called \emph{quasisimple} if $\mathbf{G}(F)$, for an algebraically closed field $F\geq K$, is non-abelian, connected, and all its Zariski closed connected normal subgroups (over $F$) are trivial. 
\end{definition}	
	
Here comes the promised result:
	
\begin{theorem}\label{thm:red_alm_smpl_cs_alg_grps}
Let $F\geq K$ be an algebraically closed field and $\mathbf{G}$ as above. Then either $\mathbf{G}(F)/\mathbf{Z}(\mathbf{G}(F))$ has a mixed identity of length at most four, or $\mathbf{G}(F)$ is simple (over the field $F$). In short: $d(\mathbf{G}/\mathbf{Z}(\mathbf{G}))=0$ or $\mathbf{G}(F)$ is simple.
\end{theorem}
	
For the proof of the theorem, we need the following two observations, which we state here without proving them.
	
\begin{fact}\label{fct:ab_cse_shrt_id}
If there is a non-trivial abelian normal subgroup $A\trianglelefteq G$, then $G$ has the mixed identity $w(x)=[a^x,b]\in G\ast\gensubgrp{x}$ for $a,b\in A\setminus\mathbf{1}$.
\end{fact}
	
\begin{fact}\label{fct:prod_cse_shrt_id}
If $G$ has two non-trivial normal subgroups $A,B\trianglelefteq G$ which centralize each other, then it satisfies the mixed identity $w(x)=[a^x,b]\in G\ast\gensubgrp{x}$, for non-trivial $a\in A \setminus\mathbf{1}$, $b \in B\setminus\mathbf{1}$, which has length four.
\end{fact}
	
\begin{proof}[Proof of Theorem~\ref{thm:red_alm_smpl_cs_alg_grps}]
We have that 
$$
\mathbf{G}(F)\subseteq\ker(\det\colon\GL(\mathbf{V}(F))\to F^\times)=\SL(\mathbf{V}(F)),
$$
since $\mathbf{G}(F)$ is perfect (and hence has no abelian quotients). Hence $\mathbf{Z}(\mathbf{G}(F))\subseteq\mathbf{Z}(\SL(\mathbf{V}(F)))$ by Schur's lemma, as $F$ is algebraically closed, and so the former is a finite cyclic group.
Write 
$$
\pi\colon\mathbf{G}(F)\twoheadrightarrow\mathbf{G}(F)/\mathbf{Z}(\mathbf{G}(F))
$$ 
for the natural morphism (an isogeny).
If $\mathbf{G}(F)$ had a non-trivial closed connected abelian normal subgroup $A$, then $A.\pi$ were a normal subgroup of the same type in $\mathbf{G}(F)/\mathbf{Z}(\mathbf{G}(F))$, and so the latter had a mixed identity of length four (by Fact~\ref{fct:ab_cse_shrt_id}). Hence we may assume that $\mathbf{G}(F)$ is a semisimple algebraic group. 
Write $\mathfrak{g}$ for the Lie algebra of $\mathbf{G}(F)$. Then $\mathfrak{g}=\bigoplus_{i\in I}{\mathfrak{g}_i^{e(i)}}$ for pairwise distinct simple Lie algebras $\mathfrak{g}_i$ and $e(i)>0$ ($i\in I$). For any $i\in I$, the subspace $\mathfrak{g}_i^{e(i)}$ is an ideal of $\mathfrak{g}$, so it generates a closed connected normal subgroup $N_i$ of $\mathbf{G}(F)$. All these $N_i$ are pairwise disjoint up to $\mathbf{Z}(\mathbf{G}(F))$. Assume there is $i\neq j$ contained in $I$; set $A\coloneqq N_i$ and $B\coloneqq N_j$. Then $[A,B]\subseteq A\cap B\leq\mathbf{Z}(\mathbf{G}(F))$.
		
Picking $a\in A\setminus\mathbf{Z}(\mathbf{G}(F))$ and $b\in B\setminus\mathbf{Z}(\mathbf{G}(F))$ (which is possible, since we are removing finitely many elements from infinite sets), the word $w(x)\coloneqq[(a.\pi)^x,b.\pi]\in\mathbf{G}(F)/\mathbf{Z}(\mathbf{G}(F))\ast\gensubgrp{x}$ is a mixed identity for $\mathbf{G}(F)/\mathbf{Z}(\mathbf{G}(F))$ of length four (by Fact~\ref{fct:prod_cse_shrt_id}). Hence we must have $\card{I}=1$. 
Then $\mathfrak{g}=\mathfrak{g}_1^{e(1)}$ and $\mathbf{G}(F)$ acts on this space. Clearly, $\mathbf{G}(F)$ must permute all copies of $\mathfrak{g}_1$ transitively, since otherwise, again there would be two non-trivial essentially disjoint connected normal subgroups. But this means there is a non-trivial homomorphism $\mathbf{G}(F)\to S_{e(1)}$, which is impossible if $e(1)\geq 2$, as $\mathbf{G}(F)$ is connected. Hence $e(1)=1$ and $\mathfrak{g}=\mathfrak{g}_1$ is simple. Thus $\mathbf{G}(F)$ is simple as a $K$-group.
\end{proof}
	
In accordance with Theorem~\ref{thm:red_alm_smpl_cs_alg_grps}, we assume now that $\mathbf{G}$ is quasisimple.
	
\subsection{Tomanov's approach}\label{subsec:tom_appr}
	
Let $L\geq K$ be a local field with absolute value $\norm{\bullet}$. Recall that this is  a map $\norm{\bullet}\colon L\to\mathbb{R}_{\geq0}$ such that $\norm{ab}=\norm{a}\norm{b}$ and $\norm{a+b}\leq\norm{a}+\norm{b}$ for all $a,b\in L$. It is called \emph{non-trivial} if there is $t\in L^\times$ such that $\norm{t}\neq 1$. Write $F=\overline{L}$ for the algebraic closure of $L$ and extend $\norm{\bullet}$ uniquely to it.
	
	
Let $\mathbf{G},\mathbf{V}$ and $K\leq L\leq \overline{L}$ and $\norm{\bullet}$ be as above. Let $e_1,\ldots,e_n\in\mathbf{V}(M)\subseteq\mathbf{V}(\overline{L})$ be a basis of decomposition for a torus $\mathbf{T}(\overline{L})$, where $M\geq L$ is a finite extension containing the splitting field of $\mathbf{T}$. Here $e_1$ is the highest weight vector and $e_n$ the lowest. Write $e_1^\ast,\ldots,e_n^\ast\in\mathbf{V}^\ast(M)$ for the dual basis.

\begin{definition}\label{def:sml_consts}
Set 
$$
\mathbf{\Lambda}(E)\coloneqq\set{c\in\GL(\mathbf{V}(E))}[(e_1.c^g).e_n^\ast=0\text{ for all }g\in\mathbf{G}(E)]
$$
to be the \emph{Tomanov-small (or T-small)} constants with respect to the field $E$, which lies between $M$ and $F=\overline{L}$.
\end{definition}
	
\begin{remark}
One observes that $\mathbf{\Lambda}(E)=\GL(\mathbf{V}(E))\cap\mathbf{\Lambda}(\overline{L})$, since $\mathbf{G}(E)$ is Zariski dense in $\mathbf{G}(\overline{L})$.
\end{remark}
	
\begin{remark}
Let $E$ be as in the previous definition. Note that $\mathbf{Z}(\GL(\mathbf{V}(E)))=E^\times 1_{\mathbf{V}(E)}\subseteq \mathbf{\Lambda}(E)$ and the latter consists of cosets of $\mathbf{Z}(\GL(\mathbf{V}(E)))$. However,  $\mathbf{G}(K)\not\subseteq \mathbf{\Lambda}(E)$,  as otherwise $\gensubsp{e_1,\ldots,e_{n-1}}_{\overline{L}}$ would be a non-trivial $\mathbf{G}(K)$-invariant subspace of  $\mathbf{V}(\overline{L})$, which is a contradiction, since $\mathbf{G}(K)$ acts (absolutely) irreducibly.
\end{remark}
	
Subsequently, we use the following approach of Tomanov~\cite{tomanov1985generalized}. Here we cite the first result on mixed identities for algebraic groups. It will be refined by Theorem~\ref{thm:sln_w_dim1} below.
	
\begin{theorem}[Tomanov~\cite{tomanov1985generalized}]\label{thm:tomanov}
There is no $\mathbf{\Lambda}(\overline{L})$-non-central identity for $\mathbf{G}(K)$ with all constants from $\GL(\mathbf{V}(\overline{L}))$.
\end{theorem}
	
	
	
In the following, we construct a semisimple element, which allows us to carry out a ping-pong like argument.
	
\begin{lemma}[\cite{tomanov1985generalized}*{Proof of Theorem~1}]\label{lem:tomanov}
Assume that $\mathbf{T}\leq\GL(\mathbf{V})$ is a maximal torus in the representation of $\mathbf{G}$.
Then there exists a finite extension $M$ of $L$ such that there is a diagonal element $g\in \mathbf{T}(M)\leq\mathbf{G}(M)\leq\GL(\mathbf{V}(M))$ with the eigenvalues $\lambda_1,\ldots,\lambda_n\in L^\times\subseteq M^\times$ (corresponding to the eigenvectors $e_1,\ldots,e_n\in \mathbf{V}(M)$, which are  chosen as prior to Definition~\ref{def:sml_consts}) such that $\norm{\lambda_1}>\norm{\lambda_i}>\norm{\lambda_n}$ for $i\in\set{2,\ldots,n-1}$.
\end{lemma}

\begin{definition}\label{def:attr_rep_pts}
Let $g\in \mathbf{G}(M)$ be as in Lemma~\ref{lem:tomanov}. Write $\mathbf{P}(\mathbf{V}(M))$ resp.\ $\mathbf{P}(\mathbf{V}(\overline{L}))$ for the projective space obtained from $\mathbf{V}(M)$ resp.\ $\mathbf{V}(\overline{L})$ and equip them with the so-called \emph{analytic} topology which is induced by the extended absolute value $\norm{\bullet}\colon M\subseteq\overline{L}\to\mathbb{R}_{\geq 0}$. Set $A(g)\coloneqq\gensubsp{e_1}_{M}\in\mathbf{P}(\mathbf{V}(M))$; $B(g)\coloneqq\gensubsp{e_1}_{\overline{L}}\in\mathbf{P}(\mathbf{V}(\overline{L}))$, and $A'(g)=\overline{\gensubsp{e_2,\ldots,e_n}}_{M}\subseteq\mathbf{P}(\mathbf{V}(M))$; $B'(g)=\overline{\gensubsp{e_2,\ldots,e_n}}_{\overline{L}}\subseteq\mathbf{P}(\mathbf{V}(\overline{L}))$. Then similarly, $A(g^{-1})=\gensubsp{e_n}_{M}$; $B(g^{-1})=\gensubsp{e_n}_{\overline{L}}$, and $A'(g^{-1})=\overline{\gensubsp{e_1,\ldots,e_{n-1}}}_{M}$; $B'(g^{-1})=\overline{\gensubsp{e_1,\ldots,e_{n-1}}}_{\overline{L}}$.
\end{definition}	
	
The following important lemma we cite here without proof.
	
\begin{lemma}[Tits]\label{lem:attrctn}
Let $U\subseteq\mathbf{P}(\mathbf{V}(M))$ be any open neighborhood of $A(g)$ (in the analytic topology) and $C\subseteq\mathbf{P}(\mathbf{V}(M))\setminus A'(g)$ be  a compact set. Then there is $N\in\mathbb{N}$ such that $C.g^n\subseteq U$ for all $n\geq N$. 
\end{lemma}

\section{Dimension of the word image}\label{chap:dim_wrd_img}

In this section we investigate the behavior of the dimension of the image of a given word map with constants for quasisimple linear algebraic groups over local fields. The following result is due to Gordeev, Kunyavski{\u\i}, and Plotkin \cite{gordeevkunyavskiiplotkin2018word}. It shows that the dimension of the image of a non-singular word map with constants is as large as possible in the generic case, and is an extension of a theorem of Borel \cite{borel1983free}.
	
\begin{theorem}[Corollary~5.3 in \cite{gordeevkunyavskiiplotkin2018word}]\label{thm:gordeev}
Recall that $F\geq K$ is algebraically closed. Let $w=w(x_1,\ldots, x_r, y_1,\ldots,y_s)\in\mathbf{F}_{r+s}$ be any non-trivial word in $r+s$ variables. If
$w(x_1,\ldots, x_r,1, . . . , 1)\neq 1_{\mathbf{F}_r}$, then there exists an non-empty Zariski open subset $U\subseteq\mathbf{G}(F)^s$ such that for every $u=(u_1,\ldots,u_s)\in U$ the word
map with constants 
$$
w(\bullet,u)\colon \mathbf{G}(F)^r\to\mathbf{G}(F)
$$ 
is dominant.
\end{theorem}
	
In the following we derive lower bounds on the Krull dimension of the word image with constants for algebraic groups, which are also valid for singular words.

\begin{remark}\label{rmk:dim_inv_isog}
Note that any constructible set $T\subseteq\SL_n(F)$ has the same dimension as its image $\overline{T}\subseteq\PSL_n(F)$, since the natural map $\SL_n(F)\twoheadrightarrow\PSL_n(F)$ is an isogeny. Hence we can work with quasisimple linear algebraic groups (with finite center) rather than simple ones (with trivial center).
\end{remark}

\subsection{Refining Tomanov's approach}

Assume that $\mathbf{\Lambda}(\overline{L})\cap I(w)=\varnothing$. Then we can find a substitution $s\colon x_i\mapsto g_{-i}x g_i$ for $g_{\pm i}\in\mathbf{G}(M)$ such that in $$
w'(x)=w(s(x_1),\ldots,s(x_r))=c_0'x^{\varepsilon(1)}c_1'\cdots c_{l-1}'x^{\varepsilon(l)} c_l'
$$ 
we have
\begin{equation}\label{eq:g_approp}
A(g^{\varepsilon(j)}).c_j'\notin A'(g^{\varepsilon(j+1)})
\end{equation}
for all $j\in\set{1,\ldots,l-1}$, where $g\in\mathbf{G}(M)$ is the element from Lemma~\ref{lem:tomanov}. We say that such an index $j$ is \emph{$g$-appropriate} for $w'$ if Equation~\eqref{eq:g_approp} holds\label{txt:g-apropp}. Also the first constant $c_0'$ and the last one $c_l'$ are not of importance when we study the dimension of the word image.
Hence we may assume w.l.o.g.\ that $w=x^{\varepsilon(1)}c_1\cdots c_{l-1}x^{\varepsilon(l)}\in\GL_n(M)\ast\gensubgrp{x}$. Moreover, by replacing $x$ by $x^{-1}$ if necessary, we may assume that $\varepsilon(l)=+1$.
	
Addition and multiplication are continuous with respect to the absolute value $\norm{\bullet}$ on $\overline{L}$. Thus any algebraic set $V(f_1,\ldots,f_n)\subseteq\overline{L}^n$ (for polynomials $f_1,\ldots,f_n\in \overline{L}[x_1,\ldots,x_m])$ is closed in the analytic topology, since it can be written as 
$$
(0,\ldots,0).(f_1,\ldots,f_n)^{-1}
$$ 
and all the $f_j$ are continuous. Hence the analytic topology is finer than the Zariski topology. Recall Definition~\ref{def:dim_img}. Also, we shall make use of Remark~\ref{rmk:incr_fld}, i.e.\ we estimate $\dim(w(\mathbf{G}(M)))$ instead of $\dim(w(\mathbf{G}(K)))$ (which are indeed equal).
	
\subsection{The special linear case}
	
We now prove a bound on the dimension of the word image for special linear groups with the method of Tomanov~\cite{tomanov1985generalized} (which only works for fields $K$ as required above).
At first we characterize Tomanov-small elements in the special linear case:
	
\begin{lemma}\label{lem:sml_elts_sln}
Recall Definition~\ref{def:sml_consts}. It holds that $\mathbf{\Lambda}(E)=\mathbf{Z}(\GL_n(E))\cong E^\times$ (when $\mathbf{G}=\SL_n$).
\end{lemma}
	
\begin{proof}
Clearly, $\mathbf{Z}(\GL_n(E))$ is contained in $\mathbf{\Lambda}(E)$. If $c\in\GL_n(E)$ is a non-central element, it has a non-zero matrix entry at position $(i,j)$ with $i\neq j$, or it has a direct summand of the form
$\left(\begin{smallmatrix}
\lambda & 0\\
0 & \mu
\end{smallmatrix}\right)$
that splits off, where $\lambda\neq\mu$ are elements of $E^\times$.
But $\left(\begin{smallmatrix}
1 & -1\\
0 & 1
\end{smallmatrix}\right)
\left(\begin{smallmatrix}
\lambda & 0\\
0 & \mu
\end{smallmatrix}\right)
\left(\begin{smallmatrix}
1 & 1\\
0 & 1
\end{smallmatrix}\right)=
\left(\begin{smallmatrix}
\lambda & \lambda-\mu\\
0 & \mu
\end{smallmatrix}\right)$
also has a non-zero off-diagonal entry. Now we may conjugate by an element of $\SL_n(E)$, so that $i=1$, $j=n$. Then $(e_1.c).e_n^\ast\neq 0$ and so $c$ is not T-small.
\end{proof}
	
Here is our first main result.
	
\begin{theorem}\label{thm:sln_w_dim1}
Let $w\in\GL_n(\overline{L})\ast\mathbf{F}_r$ be a word of length $\norm{w}=l\geq2$ such that $\mathbf{\Lambda}(\overline{L})\cap I(w)=\varnothing$. Then $\dim(w(\SL_n(K)^r))\geq 2(n-1)$.
\end{theorem}

Note that the estimate above is sharp:

\begin{example}\label{ex:conj_cls_transv_min_dim}
The word $w=x^{-1}cx=c^x$ for a \emph{transvection} $c=1_{\mathbf{V}(K)}+l v\in\SL_n(K)$ (where $v\in \mathbf{V}(K)\setminus\mathbf{0}$ and $l \in \mathbf{V}^\ast(K)\setminus\mathbf{0}$ are such that $v.l =0$) has image of dimension $2(n-1)$. Indeed, we can choose $l \in \mathbf{V}^\ast(K)\setminus\mathbf{0}$ arbitrarily up to a linear factor (i.e.\ we choose $\overline{l }\in\mathbf{P}(\mathbf{V}^\ast(K))$), which produces $n-1$ dimensions. Then we may choose any $0_{\mathbf{V}(K)}\neq v\in\ker(l)\cong K^{n-1}$ which again contributes $n-1$ dimensions. Also note that $c\notin\mathbf{\Lambda}(\overline{L})$ by Lemma~\ref{lem:sml_elts_sln}.
\end{example}
	
	
The most important application of Example~\ref{ex:conj_cls_transv_min_dim} and Theorem~\ref{thm:sln_w_dim1} is the following:
	
\begin{corollary}
Let $w$ be as in Theorem~\ref{thm:sln_w_dim1} with $L=\overline{L}=\mathbb{C}$ and $K\in\set{\mathbb{R},\mathbb{C}}$. Then we have 
$$
\dim(w(\SL_n(\mathbb{R})^r)=\dim(w(\SL_n(\mathbb{C})^r)\geq d(\SL_n)=2(n-1).
$$
\end{corollary}
	
\begin{proof}[Proof of Theorem~\ref{thm:sln_w_dim1}]
By Remark~\ref{rmk:incr_fld}, we only need to estimate the dimension of the set $w(\SL_n(M)^r)\subseteq\GL_n(M)$.
Let $e_{+1}\in \mathbf{V}(M)\setminus\mathbf{0}$ and $e_{-1}^\ast\in \mathbf{V}^\ast(M)\setminus\mathbf{0}$ be arbitrary such that $e_{+1}.e_{-1}^\ast=0$. There exists a basis $e_{+1}=e_1,\ldots,e_{-1}=e_n$ of $\mathbf{V}(M)$ such that $e_{-1}^\ast=e_n^\ast$ for the dual basis (where $e_1^\ast,\ldots,e_n^\ast\in\mathbf{V}^\ast(M)$ is the dual basis, i.e.\ $e_i.e_j^\ast=\delta_{ij}$ for all $i,j$). Now set $g=\diag(\lambda_1,\ldots,\lambda_n)\in\SL(\mathbf{V}(M))=\SL_n(M)$ to be a semisimple element with eigenbasis $e_1,\ldots,e_n$. Moreover, we require $\norm{\lambda_1}>\norm{\lambda_2}>\cdots>\norm{\lambda_{n-1}}>\norm{\lambda_n}$. Then for $(g^{-1})^\ast$ we obtain that this is the semisimple element with eigenbasis $e_1^\ast,\ldots,e_n^\ast$ acting as $\diag(\lambda_1^{-1},\ldots,\lambda_n^{-1})$. 
		
Define the set of T-small constants as in Definition~\ref{def:sml_consts}. Recall that $\mathbf{\Lambda}(\overline{L})=\mathbf{Z}(\GL_n(\overline{L}))$ by Lemma~\ref{lem:sml_elts_sln}. Recall Definition~\ref{def:attr_rep_pts}. To simplify the situation, we perform a substitution as in Lemma~\ref{lem:red_one_var_cs}. When replacing the given word 
$$
w=c_0x_{i(1)}^{\varepsilon(1)}c_1\cdots c_{l-1}x_{i(l)}^{\varepsilon(l)}c_l\in\GL_n(M)\ast\mathbf{F}_r
$$ 
without T-small (i.e.\ central; see Lemma~\ref{lem:sml_elts_sln}) critical constants by the word 
$$
w'=c_0'x^{\varepsilon(1)}c_1'\cdots c_{l-1}'x^{\varepsilon(l)}c_l'\in\GL_n(M)\ast\gensubgrp{x}
$$ 
by a substitution $x_i\mapsto g_{-i}xg_i$ for a \emph{suitable} tuple $(g_{\pm i})_{i=1}^r\in\SL_n(M)^{2r}$ depending on $g$, we have to make sure that
\begin{equation}\label{eq:nr_1}
A(g^{\varepsilon(j)}).c_j'\notin A'(g^{\varepsilon(j+1)})			
\end{equation}
and
\begin{equation}\label{eq:nr_2}
A((g^{-\varepsilon(j)})^\ast).(c_j'^{-1})^\ast\notin A'((g^{-\varepsilon(j+1)})^\ast) 
\end{equation}
for all $j\in\set{1,\ldots,l-1}$. This means that $j$ is $g$-appropriate for $w$, and $j$ is $(g^{-1})^\ast$-appropriate for $(w^{-1})^\ast$. These two conditions arise since we want to play the game of attraction and repulsion in $\mathbf{P}(\mathbf{V}(M))$ and $\mathbf{P}(\mathbf{V}^\ast(M))$ simultaneously. Equation~\eqref{eq:nr_1} is necessary to play this game with the word 
$$
w'=c_0'x^{\varepsilon(1)}c_1'\cdots c_{l-1}'x^{\varepsilon(l)}c_l'\in\GL(\mathbf{V}(M))\ast\gensubgrp{x}
$$ 
and the element $g$, and Equation~\eqref{eq:nr_2} is needed to play it with the word 
$$
(w'^{-1})^\ast=(c_0'^{-1})^\ast x^{\varepsilon(1)}(c_1'^{-1})^\ast\cdots (c_{l-1}'^{-1})^\ast x^{\varepsilon(l)}(c_l'^{-1})^\ast\in\GL(\mathbf{V}^\ast(M))\ast\gensubgrp{x}
$$ 
and the element $(g^{-1})^\ast$.
Here, in $(w'^{-1})^\ast$ the \emph{transpose-inverse} automorphism is applied only to the constants $c_j'$, but not to the variable $x$. By assumption, we have that $c_j'=g_{\varepsilon(j)i(j)}^{\varepsilon(j)}c_j g_{-\varepsilon(j+1)i(j+1)}^{\varepsilon(j+1)}$ for $j\in\set{1,\ldots,l-1}$. For the sake of readability, set $h\coloneqq g_{\varepsilon(j)i(j)}$ and $l\coloneqq g_{-\varepsilon(j+1)i(j+1)}$. Again we distinguish the following two cases:
		
\vspace{3mm} 
		
\noindent\emph{Case~1: $j\in J_0(w)\cup J_+(w)$.} Then Equation~\eqref{eq:nr_1} means that 
\begin{equation}\label{eq:nr_3}
e_{\varepsilon(j)}.h^{\varepsilon(j)}c_j l^{\varepsilon(j+1)}.e_{\varepsilon(j+1)}^\ast\neq 0.
\end{equation}
Similarly, Equation~\eqref{eq:nr_2} translates to
\begin{align}\label{eq:nr_4}
& e_{-\varepsilon(j+1)}.(e_{-\varepsilon(j)}^\ast.((h^{\varepsilon(j)}c_j l^{\varepsilon(j+1)})^{-1})^\ast)\neq 0 &\nonumber \\
& \Leftrightarrow e_{-\varepsilon(j+1)}.(h^{\varepsilon(j)}c_j l^{\varepsilon(j+1)})^{-1}.e_{-\varepsilon(j)}^\ast\neq 0.&
\end{align}
Collect the pairs $(h,l)\in\SL_n(\overline{L})^2$, which satisfy Equation~\eqref{eq:nr_3} in the set $S_j^+\subseteq\SL_n(\overline{L})^2$; and the set of such pairs, that satisfy Equation~\eqref{eq:nr_4} in $S_j^-$.
Then, since $\SL_n(\overline{L})$ acts irreducibly on $\mathbf{V}(\overline{L})$, $S_j^+$ and $S_j^-$ are non-empty and Zariski open in $\SL_n(\overline{L})^2$. Indeed, if this were not the case, $\ker(e_{\varepsilon(j+1)}^\ast)$ resp.\ $\ker(e_{-\varepsilon(j)}^\ast)$ would be non-trivial invariant subspaces under the action of $\SL_n(\overline{L})$.
Hence $S_j\coloneqq S_j^+\cap S_j^-\subseteq\SL_n(\overline{L})^2$ is again non-empty and Zariski open.
		
\vspace{3mm}
		
\noindent\emph{Case~2: $j\in J_-(w)$.} Then we have $h=l$ and $\varepsilon(j+1)=-\varepsilon(j)$. Equation~\eqref{eq:nr_1} translates to
\begin{equation}\label{eq:nr_5}
e_{\varepsilon(j)}.h^{\varepsilon(j)}c_j h^{-\varepsilon(j)}.e_{-\varepsilon(j)}^\ast\neq 0
\end{equation}
and Equation~\eqref{eq:nr_2} gives the following obstruction:
\begin{equation}\label{eq:nr_6}
e_{\varepsilon(j)}.h^{\varepsilon(j)}c_j^{-1} h^{-\varepsilon(j)}.e_{-\varepsilon(j)}^\ast\neq 0.
\end{equation}
Collect the elements $h\in\SL_n(\overline{L})$, which satisfy Equation~\eqref{eq:nr_5}, in $S_j^+\subseteq\SL_n(\overline{L})$, and the $h$'s that fulfill Equation~\eqref{eq:nr_6} in $S_j^-\subseteq\SL_n(\overline{L})$. By the assumption that $c_j\notin\mathbf{\Lambda}(\overline{L})$ is not T-small (i.e.\ non-central in $\GL_n(\overline{L})$; see Lemma~\ref{lem:sml_elts_sln}), there is an $h\in\SL_n(\overline{L})$ which satisfies Equation~\eqref{eq:nr_5} (and similarly for Equation~\eqref{eq:nr_6}, since $c_j^{-1}$ is also not T-small, as it is non-central). Hence both equations have a solution in $\SL_n(\overline{L})$. Define $S_j\coloneqq S_j^+\cap S_j^-$ as above and note that it is non-empty and Zariski open. 

Now define the projection $\pi_j\colon \mathbf{G}(\overline{L})^{2r}\to \mathbf{G}(\overline{L})^2$ by 
$$
(g_i)_{i=\pm1}^{\pm r}\mapsto (g_{\varepsilon(j)i(j)},g_{-\varepsilon(j+1)i(j+1)}).
$$ 
for $j\in J_0(w)\cup J_+(w)$ and similarly the projection $\pi_j\colon \mathbf{G}(\overline{L})^{2r}\to \mathbf{G}(\overline{L})$ by
$$
(g_i)_{i=\pm1}^{\pm r}\mapsto g_{\varepsilon(j)i(j)}
$$
for $j\in J_-(w)$. Then $T\coloneqq\bigcap_{j=1}^{l-1}S_j.\pi_j^{-1}\subseteq\SL_n(\overline{L})^{2r}$ is Zariski open. Hence, since $\SL_n(M)^{2r}$ is Zariski dense in $\SL_n(\overline{L})^{2r}$, the set $T\cap\SL_n(M)^{2r}\subseteq\SL_n(\overline{L})^{2r}$ is also Zariski dense. In particular it is non-empty, as desired, since then Equations~\eqref{eq:nr_3}, \eqref{eq:nr_4}, \eqref{eq:nr_5}, and \eqref{eq:nr_6} can be fulfilled simultaneously by a tuple from $\SL_n(M)^{2r}$.
		
\vspace{3mm}
		
In the following we write $w$ for $w'$ and $c_j$ for $c_j'$ ($j\in\set{0,\ldots,l}$). Let $v\in\mathbf{V}(M)$ and $l\in\mathbf{V}^\ast(M)$ be arbitrary such that $v.l\neq 0$. Consider the map
\begin{gather*}
\alpha\colon\SL_n(\overline{L})\to\mathbf{P}(\mathbf{V}(\overline{L}))\times\mathbf{P}(\mathbf{V}^\ast(\overline{L}));\\
g\mapsto(\overline{v.w(g)},\overline{l.(w(g)^{-1})^\ast}).
\end{gather*}
This map is regular and we can write it as the composition of the regular maps
\begin{gather*}
\beta\colon\SL_n(\overline{L})\to\GL_n(\overline{L});\\
g\mapsto w(g)
\end{gather*}
and
\begin{gather*}
\gamma\colon\GL_n(\overline{L})\to\mathbf{P}(\mathbf{V}(\overline{L}))\times\mathbf{P}(\mathbf{V}^\ast(\overline{L}));\\
g\mapsto(\overline{v.g},\overline{l.(g^{-1})^\ast}).
\end{gather*}
Note that then also
\begin{equation}\label{eq:w_pres_pring}
v.w(g).(l.((w(g)^{-1})^\ast)=v.w(g).w(g)^{-1}.l=v.l\neq 0.
\end{equation}
		
W.l.o.g. assume that $c_0=c_l=1_{\mathbf{V}(M)}$. We need that $v\notin A'(g^{\varepsilon(1)})$ and $l\notin A'((g^{-\varepsilon(1)})^\ast)$ to start our game of attraction and repulsion. This means that $v.e_{\varepsilon(1)}^\ast\neq 0$ and $e_{-\varepsilon(1)}.l\neq 0$, which are two linear Zariski open conditions on $e_{\varepsilon(1)}^\ast$ and $e_{-\varepsilon(1)}$ (where $e_{\varepsilon(1)}^\ast$ is seen as an element of $\mathbf{V}^\ast(\overline{L})\supseteq\mathbf{V}^\ast(M)$ and $e_{-\varepsilon(1)}$ as an element of $\mathbf{V}(\overline{L})\supseteq \mathbf{V}(M)$).
		
Applying Lemma~\ref{lem:attrctn}, by plugging in $g^n$ into $\alpha$, for $n$ large enough, we land in any neighborhood of $(\overline{e}_{+1},\overline{e}_{-1}^\ast)$ in the analytic topology on $\mathbf{P}(\mathbf{V}(M))\times\mathbf{P}(\mathbf{V}^\ast(M))$. Hence the closure (Zariski or analytic) of $\SL_n(M).\alpha$ contains 
$$
\set{(\overline{e}_{+1},\overline{e}_{-1}^\ast)\in\mathbf{P}(\mathbf{V}(M))\times\mathbf{P}(\mathbf{V}^\ast(M))}[e_{+1}.e_{-1}^\ast=0],
$$
since the Zariski topology is coarser than the analytic one (see the introduction of this section).
Hence its Zariski closure is
\begin{gather*}
P\coloneqq\set{(\overline{e}_{+1},\overline{e}_{-1}^\ast)\in\mathbf{P}(\mathbf{V}(\overline{L}))\times\mathbf{P}(\mathbf{V}^\ast(\overline{L}))}[e_{+1}.e_{-1}^\ast=0]\\
\subseteq\mathbf{P}(\mathbf{V}(\overline{L}))\times\mathbf{P}(\mathbf{V}^\ast(\overline{L})).
\end{gather*}
This set has $n-1+n-2=2n-3$ dimensions. Now the algebraic set
$$
\overline{\SL_n(M).\beta}^{\overline{L}}=\overline{w(\SL_n(M))}^{\overline{L}}\subseteq\SL_n(\overline{L}),
$$ 
maps onto a constructible subset $C$ of $\mathbf{P}(\mathbf{V}(\overline{L}))\times\mathbf{P}(\mathbf{V}^\ast(\overline{L}))$ under the regular map $\gamma$ (by Chevalley's theorem). The Zariski closure $\overline{C}^{\overline{L}}$ of $C$ must contain $P$, since already the closure of the image of $w(\SL_n(M))$ under $\gamma$ contains $P$. 
But $(\overline{v}',\overline{l}')\coloneqq(\overline{v.w(1)},\overline{l.(w(1)^{-1})^\ast})\notin P$ by Equation~\eqref{eq:w_pres_pring}, since $v.l\neq 0$ by assumption.
Hence we must have that 
$$
P\sqcup\set{(\overline{v}',\overline{l}')}\subset\overline{C}^{\overline{L}}\subseteq\mathbf{P}(\mathbf{V}(M))\times\mathbf{P}(\mathbf{V}^\ast(M)).
$$
However, $P$, $\overline{C}^{\overline{L}}$ and $\mathbf{P}(\mathbf{V}(M))\times\mathbf{P}(\mathbf{V}^\ast(M))$ are Zariski closed and irreducible, so that we must have that
$\overline{C}^{\overline{L}}=\mathbf{P}(\mathbf{V}(\overline{L}))\times\mathbf{P}(\mathbf{V}^\ast(\overline{L}))$ and hence 
$$
\dim(w(\SL_n(M)))\geq\dim(C)=\dim(\overline{C}^{\overline{L}})=2(n-1).
$$
This completes the proof.
\end{proof}
	
\subsection{The symplectic case}
	
The following theorem gives a sharp bound on the dimension of the word image in the symplectic group (see Example~\ref{ex:sml_conj_cls_symp} below).
	
\begin{theorem}\label{thm:dim_bds-symp}
Let $w\in\GL_n(\overline{L})\ast\mathbf{F}_r$ and of length $l\geq2$. Assume that $\mathbf{\Lambda}(\overline{L})\cap I(w)=\varnothing$ (where $\mathbf{G}=\Sp_{2m}$). Then $\dim(w(\Sp_{2m}(K))^r)\geq n=2m$.
\end{theorem}
	
\begin{proof}
By Remark~\ref{rmk:incr_fld}, we only have to prove that $\dim(w(\Sp_{2m}(M)^r)\geq n=2m$.
By Witt's lemma, for every $e_{+1}=e_1\in \mathbf{V}(M)=M^n$ (which satisfies $f(e_{+1},e_{+1})=0$), there is a map $g\in\Sp_{2m}(M)$ of the form 
$$
g=\diag(\lambda_1,\lambda_2,\cdots,\lambda_{2m-1},\lambda_{2m})
$$ 
with $\lambda_i\lambda_{2m+1-i}=1$ stabilizing the hyperbolic planes $H_i=\gensubsp{e_i}\oplus\gensubsp{e_{2m+1-i}}$ ($i\in\set{1,\ldots,m}$), and such that $\norm{\lambda_1}>\norm{\lambda_2}>\cdots>\norm{\lambda_{2m}}$. Set $e_{-1}\coloneqq e_{2m}$. Fix an arbitrary starting point $v\in \mathbf{V}(M)\setminus\mathbf{0}$. Recall Definition~\ref{def:attr_rep_pts}. We need that $v\notin A'(g^{\varepsilon(1)})$ in order to start our game of attraction and repulsion. This means that $v\notin (e_{-\varepsilon(1)})^\perp$. This is a Zariski open condition on $e_{+1}$ (or $e_{-1}$), i.e.\ the set of all such $e_{-\varepsilon(1)}$ in $\mathbf{V}(\overline{L})$ is Zariski open. In the course of the trajectory, we must also have $A(g^{\varepsilon(j)}).c_j\notin A'(g^{\varepsilon(j+1)})$, i.e.\ $j$ is $g$-appropriate, for all $j\in\set{1,\ldots,l-1}$, but this can be achieved in the same way as in the proof of Theorem~\ref{thm:sln_w_dim1}, since the critical constants $c_j$ ($j\in J_-(w)$) are not Tomanov-small (here we use the torus which contains $g$). As there we end up with a word in just one variable. Moreover, we may assume that $c_0,c_l=1_{\mathbf{V}(M)}$, and $\varepsilon(l)=+1$, i.e.\ 
$$
w=x^{\varepsilon(1)}c_1\cdots c_{l-1}x^{\varepsilon(l)}.
$$ 
Let $U\ni\overline{e}_{+1}=\overline{e}_1$ be an arbitrary open neighborhood (in the analytic topology) in $\mathbf{V}(M)\sqcup\mathbf{P}(\mathbf{V}(M))=\mathbf{P}^n(M)$. Then the sequence $(v.w(g^n))_{n\in\mathbb{N}}$ eventually lies in $U$ by Lemma~\ref{lem:attrctn}. 
Hence this sequence converges to $\overline{e}_{+1}$ in the analytic topology. 
		
Thus the closure of the subset  $v.w(\Sp_{2m}(M))\subseteq\mathbf{P}^n(M)$ in the analytic topology contains $\mathbf{P}^{n-1}(M)=\mathbf{P}(\mathbf{V}(M))$. 
But, since $M$ is infinite, we have that in the inclusion $\mathbf{P}^{n-1}(M)\subseteq\mathbf{P}^{n-1}(\overline{L})$ the first set is Zariski dense in the second. Thus $v.w(\Sp_{2m}(M))$ is Zariski dense in $\mathbf{P}^{n-1}(\overline{L})$.
		
Now the proof proceeds as the one of Theorem~\ref{thm:sln_w_dim1}: The map
\begin{gather*}
\alpha\colon\Sp_{2m}(\overline{L})\to\mathbf{V}(\overline{L})\sqcup\mathbf{P}(\mathbf{V}(\overline{L}))=\mathbf{P}^n(\overline{L});\\
g\mapsto v.w(g)
\end{gather*}
is regular, and can be written as the composition of the regular maps
\begin{gather*}
\beta\colon\Sp_{2m}(\overline{L})\to\GL_{2m}(\overline{L});\\
g\mapsto w(g)
\end{gather*}
and
\begin{gather*}
\gamma\colon\GL_{2m}(\overline{L})\to\mathbf{P}^n(\overline{L});\\
g\mapsto v.g
\end{gather*}
i.e.\ we have $\alpha=\beta\gamma$. Note that the natural codomain of $\alpha$ (which is $\mathbf{V}(\overline{L})$) is extended by $\mathbf{P}(\mathbf{V}(\overline{L}))$. This is since we want to work in the bigger space $\mathbf{P}^n(\overline{L})$ rather than $\mathbf{P}^{n-1}(\overline{L})$.
		
The algebraic set
$$
\overline{\Sp_{2m}(M).\beta}^{\overline{L}}=\overline{w(\Sp_{2m}(M))}^{\overline{L}}\subseteq\GL_n(\overline{L})
$$ 
maps onto a constructible set $C\subseteq\mathbf{V}(\overline{L})\subset\mathbf{P}^n(\overline{L})$ under the regular map $\gamma$ (see Chevalley's theorem). Then the Zariski closure $\overline{C}^{\overline{L}}$ of $C$ must contain $P\coloneqq\mathbf{P}^{n-1}(\overline{L})\subset\mathbf{P}^n(\overline{L})$, since already the Zariski closure of the image of $w(\Sp_{2m}(M))$ under $\gamma$ contains $P$ (see above). But clearly, $v'\coloneqq v.w(1)\in \mathbf{V}(M)\subseteq\mathbf{V}(\overline{L})$ is also contained in $C\subseteq\overline{C}^{\overline{L}}$. As $\overline{C}^{\overline{L}}$, $P$, and $\mathbf{P}^n(\overline{L})$ are Zariski closed and irreducible and $P\sqcup\set{v'}\subset\overline{C}^{\overline{L}}\subseteq\mathbf{P}^n(\overline{L})$, we must have $\mathbf{P}^n(\overline{L})=\overline{C}^{\overline{L}}$ and hence
$\dim(w(\Sp_{2m}(M)))\geq\dim(C)=\dim(\overline{C}^{\overline{L}})=n$.
The proof is complete.
\end{proof}
	
\begin{example}\label{ex:sml_conj_cls_symp}
Let $w=x^{-1}cx=c^x$, where $c\coloneqq t_v\colon x\mapsto x+f(x,v)v$ is a \emph{symplectic transvection} ($v\in \mathbf{V}(K)\setminus\mathbf{0}$). Then $w\in\Sp_{2m}(K)\ast\gensubgrp{x}$ has image of dimension $n=2m$. Indeed, for any $u\in \mathbf{V}(K)\setminus\mathbf{0}$ the transvections $t_u$ and $t_v$ are conjugate by an element $g\in\Sp_{2m}(K)=\Sp(\mathbf{V}(K))$ with $u.g=v$ which exists by Witt's Lemma: $x.t_u^g=x+f(x.g^{-1},u)u.g=x+f(x,u.g)u.g=x+f(x,v)v=x.t_v$.	Moreover, $t_u=t_v$ if and only if $u=\pm v$, so that, when $T$ is the conjugacy class of a symplectic transvection, then $(\mathbf{V}(K)\setminus\mathbf{0})/\set{\pm1}\cong T$ is a $\PSp_{2m}(K)$-variety via the map $\set{\pm v}\mapsto t_v$. Hence $\dim(\im(w))=\dim(T)=\dim(\mathbf{V}(K))=n$. Moreover, $c\notin\Lambda(\overline{L})$, since $e_1.t_{e_{2m}}=e_1+e_{2m}\notin\ker(e_{2m}^\ast)$.
\end{example}
	
\subsection{The orthogonal case}
	
Assume that $Q\colon \mathbf{V}(K)\to K$ is a quadratic form on $\mathbf{V}(K)=K^n$ and let  $L$ be as before. Then, by Lemma~\ref{lem:tomanov}, there is a finite extension $M\geq L$ such that $Q$ has an isotropic vector $v\in\mathbf{V}(M)$, i.e.\ $Q(v)=0$.
	
\begin{theorem}
Let $w\in\GO(\mathbf{V}(\overline{L}),Q)\ast\mathbf{F}_r$ be a word of length $\norm{w}=l\geq2$. Assume that $\mathbf{\Lambda}(\overline{L})\cap I(w)=\varnothing$ (see Definition~\ref{def:sml_consts} for $\mathbf{G}=\SO_m$). Then 
$$
\dim(w(\SO(\mathbf{V}(K),Q)^r)\geq m-1.
$$
\end{theorem}
	
\begin{proof}
Let $e_{+1}\in \mathbf{V}(M)\setminus\mathbf{0}$ be an isotropic vector. Then $V_M(Q)\setminus e_{+1}^\perp\neq\varnothing$, so there exists $e_{-1}\neq 0$ which is contained in this set, i.e.\ it is isotropic but $f(e_{+1},e_{-1})\neq 0$ and hence we can w.l.o.g.\ assume that $f(e_{+1},e_{-1})=1$. Then $\gensubsp{e_{+1}}\oplus\gensubsp{e_{-1}}$ is a hyperbolic plane that splits off. Hence there is a matrix $g=\diag(\lambda,m,\lambda^{-1})\in\SO(\mathbf{V}(M),Q)$, with an othogonal matrix $m$, such that $e_\varepsilon$ is an eigenvector with eigenvalue $\lambda^\varepsilon$ ($\varepsilon\in\set{\pm1}$). We can choose $\norm{\lambda}$ large enough so that $\overline{e}_{+1}$ becomes the unique point of attraction and $\overline{e}_{-1}$ the unique point of repulsion in $\mathbf{P}(\mathbf{V}(M))$.
As before (see the proof of Theorem~\ref{thm:sln_w_dim1}), we can assume that our word $w$ has only one variable $x$ and that every index $j\in\set{1,\ldots,l-1}$ is $g$-appropriate in $w=x^{\varepsilon(1)}c_1\cdots c_{l-1}x^{\varepsilon(l)}$ and $\varepsilon(l)=+1$. We can neglect the constants $c_0$ and $c_l$. Let $v\in V_M(Q)\setminus\mathbf{0}$ be the starting vector of our word $w$. Then we need that $v\notin A'(g^{\varepsilon(1)})$ (see Definition~\ref{def:attr_rep_pts} for this notation). This means that $v\not\perp e_{-\varepsilon(1)}$ which is a linear condition on $e_{+1}$ or $e_{-1}$. The game of attraction and repulsion gives that the sequence $(v.w(g^n))_{n\in\mathbb{N}}$, all of whose entries lie in $\mathbf{V}(M)$, converges in $\mathbf{P}^n(M)=\mathbf{V}(M)\sqcup\mathbf{P}(\mathbf{V}(M))$ to $\overline{e}_{+1}\in\overline{V_M(Q)\setminus\mathbf{0}}\subset\mathbf{P}(\mathbf{V}(M))$.
Hence, since $e_{+1}\in V_M(Q)$ was almost arbitrary (up to the linear condition above),  the analytic closure of $v.w(\SO(\mathbf{V}(M),Q))\subseteq\mathbf{P}^n(\overline{L})$ contains all points from $\overline{V_M(Q)\setminus\mathbf{0}}$. The same must hold for the Zariski closure as it is coarser than the analytic one. As $Q$ is a quadratic form, $\overline{V_M(Q)\setminus\mathbf{0}}$ is Zariski dense in $P\coloneqq\overline{V_{\overline{L}}(Q)\setminus\mathbf{0}}$. Alternatively, this holds, since $P=\overline{v}.\SO(\mathbf{V}(\overline{L}),Q)$ is a flag variety (of the flag $\mathbf{0}\lessdot\gensubsp{v}_{\overline{L}}<\mathbf{V}(\overline{L})$), and hence a complete variety, in which the $M$-points $\overline{v}.\SO(\mathbf{V}(M),Q)=\overline{V_M(Q)\setminus\mathbf{0}}$ are dense. Hence 
\begin{equation}\label{eq:dnse1}
P\subseteq\overline{v.w(\SO(\mathbf{V}(M),Q))}^{\overline{L}}\subseteq\mathbf{P}^n(\overline{L}).
\end{equation}
		
We proceed as in the proof of the previous theorem: Consider the regular map
\begin{gather*}
\alpha\colon\SO(\mathbf{V}(\overline{L}),Q)\to\mathbf{V}(\overline{L})\sqcup\mathbf{P}(\mathbf{V}(\overline{L}));\\
g\mapsto v.w(g)
\end{gather*}
which is a composition of the regular maps
\begin{gather*}
\beta\colon\SO(\mathbf{V}(\overline{L}),Q)\to\GO(\mathbf{V}(\overline{L}),Q);\\
g\mapsto w(g)
\end{gather*}
and
\begin{gather*}
\gamma\colon\GO(\mathbf{V}(\overline{L}),Q)\to\mathbf{V}(\overline{L})\sqcup\mathbf{P}(\mathbf{V}(M));\\
g\mapsto v.g.
\end{gather*}
Let $C$ be the image of $\overline{w(\SO(\mathbf{V}(M),Q))}^{\overline{L}}\subseteq\GO(\mathbf{V}(\overline{L}),Q)$ under $\gamma$. Then $C\subseteq V_{\overline{L}}(Q)\subset\mathbf{V}(\overline{L})$, since $Q(v.w(g))=Q(v)=0$ holds for all $g\in\SO(\mathbf{V}(\overline{L}),Q)$.
Moreover, the Zariski closure $\overline{C}^{\overline{L}}\subseteq\mathbf{P}^n(\overline{L})$ of $C$ contains $P$, since even the Zariski closure of the image of $w(\SO(\mathbf{V}(M),Q))$ under $\gamma$ contains $P$ (by Equation~\eqref{eq:dnse1}). Set $v'\coloneqq v.w(1)\in V_M(Q)$ and note that $P\sqcup\set{v'}\subset \overline{C}^{\overline{L}}\subseteq V_{\overline{L}}(Q)\sqcup P$. But since $P$, $\overline{C}^{\overline{L}}$, and $V_{\overline{L}}(Q)\sqcup P$ are all  Zariski closed and irreducible, we must have that $\overline{C}^{\overline{L}}=V_{\overline{L}}(Q)\sqcup P$.
		
Hence we obtain 
\begin{align*}
\dim(w(\SO(\mathbf{V}(M),Q))\geq\dim(C)=\dim(\overline{C}^{\overline{L}})\\
=\dim(V_{\overline{L}}(Q)\sqcup P)=\dim(V_{\overline{L}}(Q))=m-1.
\end{align*}
The proof is complete.
\end{proof}
	
\begin{example}
The \emph{orthogonal transvections} are of the form $t_v\colon x\mapsto x-f(x,v)v$ for $v\in \mathbf{V}(K)$ of norm $Q(v)=1$. We may assume w.l.o.g.\ that such a $v$ exists by scaling $Q$ appropriately. Then $t_v\in\GO(\mathbf{V}(K),Q)$ has determinant $-1$. Since $Q$ is a quadratic form, $V_K(Q-Q(v))$ is Zariski dense in $V_{\overline{L}}(Q-Q(v))$, so it has dimension $n-1$. Moreover, $\SO(\mathbf{V}(K),Q)$ acts transitively on all such vectors $v\in \mathbf{V}(K)$ of norm one by Witt's Lemma, and we have $x.t_v^g=x-f(x.g^{-1},v)v.g=x-f(x,v.g)v.g=x.t_{v.g}$. One verifies that $t_u=t_v$ if and only if $u=\pm v$. Set $c\coloneqq t_v$ and $w\coloneqq c^x=x^{-1}cx\in\GO(\mathbf{V}(K),Q)\ast\gensubgrp{x}$. Let $T=t_v^{\SO(\mathbf{V}(K),Q)}$ for an arbitrary $v\in V_K(Q-1)$. When $n=\dim(\mathbf{V}(K))$ is odd, $-T$ is the conjugacy class of $-t_v$ in $\SO(\mathbf{V}(K),Q)$, so that even $w'=(-c)^x\in\SO(\mathbf{V}(K),Q)\ast\gensubgrp{x}$. Generally, $\im(w)=\im(w')=V_K(Q-1)/\set{\pm1}\cong T$ is of dimension $n-1$, too. Here $c=t_v\notin\Lambda(\overline{L})$, since $e_1.t_{e_{2m}}=e_1-e_{2m}\notin\ker(e_{2m}^\ast)$.
\end{example}
	
\subsection{The exceptional groups of Lie type}
    
\begin{definition}
Subsequently, a vector $e_{+1}\in \mathbf{V}(M)\setminus\mathbf{0}$ is called \emph{appropriate} if there is an element $g\in\mathbf{G}(M)$ and a vector $e_{-1}\in \mathbf{V}(M)\setminus\mathbf{0}$ such that $\overline{e}_{+1}$ is the unique point of attraction and $\overline{e}_{-1}$ is the unique point of repulsion of $g$ acting on the projective space $\mathbf{P}(\mathbf{V}(M))$.
\end{definition}
	
Let $L$ be as above. Then there is a finite extension $M\geq L$ such that the action of $\mathbf{G}(M)$ on $\mathbf{V}(M)$ has an appropriate vector $e_{+1}\in\mathbf{V}(M)$ with respect to the element $g\in\mathbf{G}(M)$. This follows from Lemma~\ref{lem:tomanov}.
	
\begin{remark}
Clearly, every appropriate vector is isotropic. Here, a vector $v\in\mathbf{V}(M)\setminus\mathbf{0}$ is called \emph{isotropic} if there is $\lambda\in M^\times$ such that $\norm{\lambda}>1$, and a group element $g\in\mathbf{G}(M)$, so that $v.g=\lambda v$. 
\end{remark}
	
\begin{definition}
The minimal dimension of a homogeneous space of $\mathbf{G}$ is denoted by $\delta(\mathbf{G})$.
\end{definition}
	
The following theorem gives a bound on the dimension of the word image which coincides with the previous bounds for all classical groups except type $A_n$.

\begin{theorem}\label{thm:dim_except_grps}
Let $w\in\mathbf{G}(\overline{L})\ast\mathbf{F}_r$ be a word without Tomanov-small critical constants (i.e.\ $\mathbf{\Lambda}(\overline{L})\cap I(w)=\varnothing$). Then $\dim(w(\mathbf{G}(K)^r))\geq\delta(\mathbf{G})+1$.
\end{theorem}
	
\begin{proof}		
Again, we can assume that $w$ has only one variable and every index $j\in\set{1,\ldots,l-1}$ is $g$-appropriate in 
$$
w=x^{\varepsilon(1)}c_1'\cdots c_{l-1}'x^{\varepsilon(l)}
$$ 
with $\varepsilon(l)=+1$; here the $c_j'$ depend on the different $g\in\mathbf{G}(K)$ which we plug into $w$, and $c_0',c_l'$ can be ignored.
Let $e_{+1}\in\mathbf{V}(M)$ be an appropriate vector which exists by Lemma~\ref{lem:tomanov} for an appropriate finite extension $M\geq K$. 
We want to estimate the dimension of the set
$$
\overline{w(\mathbf{G}(K))}^{\overline{L}}=\overline{w(\mathbf{G}(L))}^{\overline{L}}=\overline{w(\mathbf{G}(M))}^{\overline{L}}=\overline{w(\mathbf{G}(\overline{L}))}^{\overline{L}}\subseteq\GL(\mathbf{V}(\overline{L})),
$$
the equality following from Remark~\ref{rmk:incr_fld}.
Hence we may concentrate on the term $\overline{w(\mathbf{G}(M))}^{\overline{L}}$.
		
Note that the property of being appropriate or isotropic is $\mathbf{G}(M)$-invariant: when $e_{+1}$ is appropriate, then $e_{+1}.h$ is appropriate as well (for all $h\in\mathbf{G}(M)$). The same holds for the property of being isotropic. This is, since $e_{+1}.h.(g^h)=e_{+1}.gh=\lambda e_{+1}.h$ for suitable $\lambda\in M^\times$. 
		
In order to start the game of attraction and repulsion, we need a starting point $v\in\mathbf{V}(M)\setminus\mathbf{0}$ in the orbit 
$e_{+1}.\mathbf{G}(M)$ of $e_{+1}$ such that $v\notin A'(g^{\varepsilon(1)})\eqqcolon H\subset\mathbf{V}(M)$. If this does not hold, then we can replace $g$ by $g'\coloneqq g^h$ (and thus $e_{+1}$ and $e_{-1}$ by $e_{+1}'\coloneqq e_{+1}.h$ and $e_{-1}'\coloneqq e_{-1}.h$, respectively) for arbitrary $h\in\mathbf{G}(M)$ such that $v\notin H.h$ (which is a Zariski open condition in $h$, which is non-empty, since $\mathbf{G}(M)$ acts (absolutely) irreducibly).
		
As before, the sequence $(v.w(g^n))_{n\in\mathbb{N}}$ converges in $\mathbf{V}(M)\sqcup\mathbf{P}(\mathbf{V}(M))$ to the point $\overline{e}_{+1}\in\mathbf{P}(\mathbf{V}(M))$. Hence the analytic closure (and so the Zariski closure) of $v.w(\mathbf{G}(M))\subseteq \mathbf{V}(M)\sqcup\mathbf{P}(\mathbf{V}(M))$ contains all of $\overline{e}_{+1}.\mathbf{G}(M)=\overline{v}.\mathbf{G}(M)$. Hence
\begin{equation}\label{local field}
\overline{v}.\mathbf{G}(M)\subseteq\overline{v}.\mathbf{G}(\overline{L})\subseteq \overline{v.w(\mathbf{G}(M))}^{\overline{L}}\subseteq \mathbf{V}(\overline{L})\sqcup\mathbf{P}(\mathbf{V}(\overline{L})).
\end{equation}
		
Set
$$
C\coloneqq v.\overline{w(\mathbf{G}(M))} ^{\overline{L}}\subseteq\mathbf{V}(\overline{L})
$$ 
This is a constructible set, since it is the image of the Zariski closed set $\overline{w(\mathbf{G}(M))} ^{\overline{L}}$ under the regular map $\gamma\colon \mathbf{G}(\overline{L})\to\mathbf{V}(\overline{L}); g\mapsto v.g$. Then $C\subseteq v.\mathbf{G}(\overline{L})$. Its Zariski closure $\overline{C}^{\overline{L}}\subseteq\mathbf{V}(\overline{L})\sqcup\mathbf{P}(\mathbf{V}(\overline{L}))$ contains $\overline{v}.\mathbf{G}(\overline{L})$, as even $\overline{v}.\mathbf{G}(\overline{L})\subseteq\overline{v.w(\mathbf{G}(M))}^{\overline{L}}\subseteq\mathbf{V}(\overline{L})\sqcup\mathbf{P}(\mathbf{V}(\overline{L}))$ (by Equation~\eqref{local field} above). Let $v'\coloneqq v.w(1)\in v.w(\mathbf{G}(M))$. Then 
$$
\overline{v}.\mathbf{G}(\overline{L})\sqcup\set{v'}\subset \overline{C}^{\overline{L}}\subseteq \overline{v}.\mathbf{G}(\overline{L})\sqcup\bigcup\overline{v}.\mathbf{G}(\overline{L})\subseteq\mathbf{P}(\mathbf{V}(\overline{L}))\sqcup\mathbf{V}(\overline{L}) .
$$ 
However, $\overline{v}.\mathbf{G}(\overline{L})$ is Zariski closed and irreducible, since it is a flag variety and a regular image of the connected (i.e.\ irreducible) $K$-group $G$. The same holds for the sets $\overline{C}^{\overline{L}}$ and $\overline{v}.\mathbf{G}(\overline{L})\sqcup\bigcup\overline{v}.\mathbf{G}(\overline{L})$. Altogether, this implies that $\overline{C}^{\overline{L}}=\overline{v}.\mathbf{G}(\overline{L})\sqcup\bigcup\overline{v}.\mathbf{G}(\overline{L})$.
Thus we get
\begin{gather*}
\dim(w(\mathbf{G}(M)))\geq\dim(C)=\dim(\overline{C}^{\overline{L}})\\
=\dim(\overline{v}.\mathbf{G}(\overline{L})\sqcup\bigcup\overline{v}.\mathbf{G}(\overline{L}))=\dim(\overline{v}.\mathbf{G}(\overline{L}))+1\geq\delta(\mathbf{G})+1
\end{gather*}
This completes the proof.	
\end{proof}


For an algebraic group $G$, define
$$
\mu(G)\coloneqq\min_{g\notin\mathbf{Z}(G)}\dim(g^G).
$$
We then have the following table \cites{josephminimalorbit,josephminimalrealizations}:

\begin{table}[htb]
\centering
\begin{tabular}{|c|c|c|c|}
\hline
$\mathbf{G}$ & $\dim(\mathbf{G})$ & $\delta(\mathbf{G})$ & $\mu(\mathbf{G})$ \\
\hline
$A_n$ $(n\geq 1)$ & $n^2+2n$ & $n$ & $2n$ \\
$B_n$ $(n\geq 2)$ & $n(2n+1)$ & $2n-1$ & $2n$ \\
$C_n$ $(n\geq 2)$ & $n(2n+1)$ & $2n-1$ & $2n$ \\
$D_n$ $(n\geq 4)$ & $n(2n-1)$ & $2n-2$ & $4n-6$ \\
\hline
$G_2$ & $14$  & $5$  & $6$  \\
$F_4$ & $52$  & $15$ & $16$ \\
$E_6$ & $78$  & $16$ & $22$ \\
$E_7$ & $133$ & $27$ & $34$ \\
$E_8$ & $248$ & $57$ & $58$ \\
\hline
\end{tabular}
\caption{Minimal non-discrete homogeneous space imensions $\delta(G)$ and minimal non-central conjugacy class dimensions $\mu(G)$ for simple groups of irreducible Dynkin type}
\label{tab:dynkin_minimal}
\end{table}


This will show that our estimates are sharp possibly apart from the cases $D_n$, $E_6$ and $E_7$. We only have to check that one can realize the dimension by non-Tomanov-small elements.

Let $u_\theta$ be a long-root element, let $V(\theta)=\mathfrak g$
be the adjoint module, and let $r_v$ denote an orthogonal reflection.
For the usual representations realizing $\delta(\mathbf G)$ one has (see \cite{plotkinsemenovvavilov1998atlas}, \cite{springerveldkamp2000octonions}*{Chs.~2, 5 and
\S\S7.2--7.3}, \cites{springer2006groups,vavilovluzgarevpevzner2007e6,vavilovluzgarev2012e7}):
\[
\begin{array}{|c|c|c|c|c|}
\hline
\text{type and module} & c & \dim(c^{\mathbf G})
 & \text{Gordeev-small} & c\in\mathbf\Lambda_V \\ \hline
A_n,\ V(\omega_1) & u_\theta & 2n
 & \text{no} & \text{no}\\
B_n,\ V(\omega_1) & -r_v & 2n
 & \text{yes (semisimple)} & \text{no}\\
C_n,\ V(\omega_1) & u_\theta & 2n
 & \text{yes (unipotent)} & \text{no}\\
D_n,\ V(\omega_1) & u_\theta & 4n-6
 & \text{no} & \text{yes}\\\hline
G_2,\ V(\omega_1) & s,\ C_{\mathbf G}(s)\simeq A_2 & 6
 & \text{yes (semisimple)} & \text{no}\\
G_2,\ V(\theta) & u_\theta & 6
 & \text{yes (unipotent)} & \text{no}\\
F_4,\ V({\omega_4}) & s,\ C_{\mathbf G}(s)\simeq B_4 & 16
 & \text{yes (semisimple)} & \text{no}\\
F_4,\ V(\theta) & u_\theta & 16
 & \text{yes (unipotent)} & \text{no}\\
E_6,\ V(\omega_1)\text{ or }V({\omega_6}) & u_\theta & 22
 & \text{no} & \text{yes}\\
E_7,\ V({\omega_7}) & u_\theta & 34
 & \text{no} & \text{yes}\\
E_8,\ V(\theta) & u_\theta & 58
 & \text{no} & \text{no}\\\hline
\end{array}
\]
Here $s$ has order $3$ in type $G_2$ and order $2$ in type $F_4$.
In the short-root modules $V(\omega_1)$ for $G_2$ and
$V({\omega_4})$ for $F_4$, the alternative minimal representative
$u_\theta$ is T-small; the semisimple class displayed above supplies
the required non-T-small representative.

Indeed, in the T-small rows the relevant long-root operator has square
zero and the highest and lowest weights do not differ by a root, so
every extreme matrix coefficient vanishes.  In the adjoint module,
on the other hand,
\[
 e_\theta\cdot x_{-\theta}(t)
   =e_\theta+\cdots-t^2e_{-\theta},
\]
so a long-root element is not T-small. Consequently the conjugacy-class examples prove sharpness for the
displayed bounds in types $A_n,B_n,C_n,G_2,F_4,E_8$.

\end{document}